\documentclass{amsart}
\usepackage{amsthm}
\usepackage{amssymb}
\usepackage{amsxtra}     
\usepackage{epsfig}
\usepackage{verbatim}

\theoremstyle{plain}
\newtheorem{theorem}{Theorem}
\newtheorem{lemma}{Lemma}

\def\bfalpha{\boldsymbol{\alpha}}
\def\bfbeta{\boldsymbol{\beta}}
\def\bfgamma{\boldsymbol{\gamma}}

\def\bfa{\boldsymbol{a}}
\def\bfw{\boldsymbol{w}}
\def\bfx{\boldsymbol{x}}
\def\bfy{\boldsymbol{y}}
\def\bfz{\boldsymbol{z}}

\def\sgn{\textrm{sgn}}

\include{header}

\begin{document}
\bibliographystyle{plain}

\title[On solution-free sets for simultaneous polynomials]{On solution-free sets for simultaneous diagonal polynomials}
\author{Matthew L. Smith}
\date{}
\address{Department of Mathematics\\
University of British Columbia\\
1984 Mathematics Drive \\
Vancouver, BC V6T 1Z2, CANADA}
\email{msmith@math.ubc.ca}
\thanks{The author was partially supported by NSF grant DMS-0601367.}

\begin{abstract}
We consider a translation and dilation invariant system consisting of $k$ diagonal equations of degrees $1,2,\ldots,k$ with integer coefficients in $s$ variables, where $s$ is sufficiently large in terms of $k$.  We show via the Hardy-Littlewood circle method that if a subset $\mathcal{A}$ of the natural numbers restricted to the interval $[1,N]$ satisfies Gowers' definition of uniformity of degree $k$, then it furnishes roughly the expected number of simultaneous solutions to the given equations.  If $\mathcal{A}$ furnishes no non-trivial solutions to the given system, then we show that the number of elements in $\mathcal{A}\cap[1,N]$ grows no faster than a constant multiple of $N/(\log\log N)^{-c}$ as $N\rightarrow\infty$, where $c>0$ is a constant dependent only on $k$.  In particular, we show that the density of $\mathcal{A}$ in $[1,N]$ tends to 0 as $N$ tends to infinity.
\end{abstract}

\maketitle

\section{Introduction}
Having considered the problem of solution-free sets for a system consisting of a quadratic equation and a linear equation in \cite{smith}, we now consider the system of $k\geq 2$ equations
\begin{equation}
L_j(x_1,\ldots,x_s) = \lambda_1 x_1^j+\cdots+\lambda_s x_s^j = 0\qquad(1\leq j\leq k).
\label{equation1}
\end{equation}
The coefficients $\lambda_1,\ldots,\lambda_s\in\mathbb{Z}$ are fixed and satisfy
\begin{displaymath}
\lambda_1+\cdots+\lambda_s=0,
\end{displaymath}
ensuring that the system (\ref{equation1}) is translation and dilation invariant with respect to the $x_i$, as the reader may easily verify.  For any set $\mathcal{A}\subset\mathbb{N}$, we define
\begin{displaymath}
\mathcal{A}_N=\mathcal{A}\cap[1,N],\quad\quad\delta_N=|\mathcal{A}_N|/N.
\end{displaymath}
As in \cite{smith}, we are interested in the upper density $\lim\sup_{N\rightarrow\infty}\delta_N$ of a set $\mathcal{A}$ furnishing no non-trivial solutions to the system (\ref{equation1}), and in particular whether or not this upper density is zero for such a set.

As in the case $k=2$, for any $N\in\mathbb{N}$, we may obtain solutions to (\ref{equation1}) in $\mathcal{A}_N^s$ by setting $x_1=x_2=\ldots=x_s$.  If further the set of coefficients $\{\lambda_1,\ldots,\lambda_s\}$ may be partitioned into $r$ non-intersecting sets
\begin{displaymath}
\left\{\lambda_{i,1},\ldots,\lambda_{i,\rho(i)}\right\}\qquad(1\leq i\leq r),
\end{displaymath}
such that $\rho(1)+\ldots+\rho(r)=s$ and
\begin{equation}
\lambda_{i,1}+\ldots+\lambda_{i,\rho(i)}=0\qquad (1\leq i\leq r),
\label{equation6}
\end{equation}
then we may generate additional solutions by setting $x_{i,1}=\ldots=x_{i,\rho(i)}$ for $1\leq i\leq r$.  We refer to such solutions as trivial solutions.  We note that as the set of coefficients may be partitioned into at most $s/2$ sets of two or more elements, and that at most $[s/2]!$ such partitions exist such that each subset satisfies the criterion (\ref{equation6}), the number of trivial solutions is at most $[s/2]! |\mathcal{A}_N|^{s/2}$.

As in \cite{smith}, our goal is twofold.  First, we wish to find an upper bound for the density $\delta_N$ of a maximal subset $\mathcal{A}_N \subset [1,N]$ which furnishes no non-trivial solutions in $\mathcal{A}_N^s$ to the system (\ref{equation1}).  Second, if instead $\mathcal{A}_N$ does furnish non-trivial solutions to the system (\ref{equation1}), then we seek to obtain a lower bound for the number $\mathcal{N}$ of solutions which is consistent with the anticipated product of local densities.

We first define
\begin{equation}
s_0(k)=2k\left[k(\log k+2\log\log k)\right]+10k^2+6.
\label{equation37}
\end{equation}
We note that the quantity $s_0(k)/2$ is the minimum number of variables currently required to establish the asymptotic bound in Vinogradov's mean value theorem (see, for example, $\S$3 of \cite{wooley}).

We also define a \textit{non-singular solution} to the system (\ref{equation1}) to be a solution $\bfa=(a_1,\ldots,a_s)$ such that there exists at least one subset $\{i_1,\ldots,i_k\}$ of the indices $\{1,\ldots,s\}$ such that
\begin{equation}
\Delta\left(i_1,\ldots,i_k\right)=\left.\det\left(\frac{\partial L_j}{\partial x_{i_l}}\right)\right|_{\bfa}\neq 0,
\label{equation40}
\end{equation}
that is, the Jacobian associated to the variables $x_{i_1},\ldots,x_{i_k}$ is non-zero when evaluated at $\bfa$.  We note that since
\begin{displaymath}
\Delta\left(i_1,\ldots,i_k\right)=k!\lambda_{i_1}\cdots\lambda_{i_k}\prod_{1\leq u<v\leq k}(x_{i_u}-x_{i_v}),
\end{displaymath}
the solution $\bfa$ is non-singular if the $a_i$ take on at least $k$ different values.

\begin{theorem}
Suppose that $s>s_0(k)$, and that $\mathcal{A}^s$ contains no non-trivial solutions to the system (\ref{equation1}).  Suppose further that the system (\ref{equation1}) possesses both a non-singular real solution and a non-singular $p$-adic solution for all rational primes $p$.  Then for $N$ sufficiently large in terms of the $\lambda_i$, there exists a constant $c>0$ dependent on $k$ such that $\delta_N\ll(\log\log N)^{-c}$.
\label{thm1}
\end{theorem}

We show in our proof of Theorem \ref{thm1} that we may take $c=2^{-2^{k+9}}$.

\begin{theorem}
Suppose that $s>s_0(k)$, and suppose that the system (\ref{equation1}) possesses both a non-singular real solution and a non-singular $p$-adic solution for all rational primes $p$.  There exists a constant $K$, dependent at most on the $\lambda_i$, with the following property.  Suppose the set $\mathcal{A}_N$ has cardinality $\delta_N N$ and is $\mathfrak{a}$-uniform of degree $k$, where the parameter $\mathfrak{a}$ obeys the upper bound
\begin{equation}
\mathfrak{a}\leq K \delta_N^{2^{k+1}(s_0(k)+2)}.
\label{equation34}
\end{equation}
Then for $N$ sufficiently large in terms of the $\lambda_i$, the set $\mathcal{A}_N^s$ necessarily contains non-trivial solutions to the system $L_k=\ldots=L_1=0$, and the number $\mathcal{N}$ of solutions in $\mathcal{A}_N^s$ satisfies the lower bound
\begin{equation}
\mathcal{N}\gg\delta_N^s N^{s-k(k+1)/2},
\label{equation35}
\end{equation}
where the implicit constant is dependent at most on the $\lambda_i$ and $k$.
\label{thm2}
\end{theorem}

We note that if $\delta_N\ll N^{-1+k(k+1)/s}$, then one can show that the trivial solutions alone contribute more than the lower bound given in (\ref{equation35}).  We note also that if $s$ is sufficiently large in terms of $k$, then we anticipate that, in the presence of suitable conditions on the $\lambda_i$, the local solubility hypotheses may be removed from the statements of Theorems \ref{thm1} and \ref{thm2}.

We recall that Erd\H{o}s and Tur\'{a}n \cite{erdosturan} conjectured in 1936 that a set $\mathcal{A}\subset\mathbb{N}$ with positive upper density necessarily contains arbitrarily long arithmetic progressions, or, equivalently, solutions in $\mathcal{A}^k$ to the system of $k-2$ translation and dilation invariant linear equations
\begin{displaymath}
x_j-2x_{j+1}+x_{j+2}=0\qquad (1\leq j\leq k-2)
\end{displaymath}
for any $k\geq 3$.  In 1953, Roth \cite{roth1} used a variation of the Hardy-Littlewood circle method to prove that a set $\mathcal{A}$ containing no three-term progressions satisfies $\delta_N\ll(\log\log N)^{-1}$.  The methods of Roth's proof have been successively refined by Heath-Brown \cite{h-b}, Szemer\'{e}di \cite{szem3}, and Bourgain \cite{bourgain2}, \cite{bourgain3}, the latter of whom has obtained the current strongest bound of $\delta_N\ll(\log\log N)^2(\log N)^{-2/3}$.

Szemer\'{e}di proved the Erd\H{o}s-Tur\'{a}n conjecture for the case $k=4$ in 1969 \cite{szem1} via a combinatorial argument.  He followed this result with an elementary proof of the conjecture for general $k$ in 1975 \cite{szem2}, although neither proof resulted in a reasonable explicit upper bound for $\delta_N$ as $N\rightarrow\infty$.  It was not until 2001 that Gowers \cite{gowers2}, using the method of exponential sums, was able to give the first explicit upper bound of $\delta_N\ll(\log\log N)^{-c}$, where $c$ is a positive constant dependent on $k$.  The crucial new device in Gowers' proof is a new notion of pseudorandomness based on polynomial uniformity of degree $d$.  Gowers showed that a set which is sufficiently pseudorandom under this new definition necessarily contains an arithmetic progression of length $k$, while a set which is not sufficiently pseudorandom may be shown to be unusually dense in some long arithmetic progression $\mathcal{P}$.

As well as obtaining an upper bound on $\delta_N$ when $\mathcal{A}_N^s$ contains no non-trivial solutions to the system (\ref{equation1}), we are also interested in obtaining a lower bound on the number of solutions $\mathcal{N}$ in $\mathcal{A}_N^s$ to the system (\ref{equation1}) when the uniformity parameter $\mathfrak{a}$ is sufficiently small in terms of the density $\delta_N$.  To this end, we employ some of the same techniques used by Vinogradov \cite{vin2} and Hua \cite{hua2} to show that when $s$ is sufficiently large in terms of $k$, the number of solutions $\mathcal{N}$ in $[1,N]^{2s}$ to the translation and dilation invariant system
\begin{displaymath}
x_1^j+\ldots+x_s^j=y_1^j+\ldots+y_s^j\qquad (1\leq j\leq k)
\end{displaymath}
satisfies $\mathcal{N}\ll N^{2s-k(k+1)/2}$, where the implied constant depends on $k$ and $s$.  This problem is associated with Vinogradov's mean value theorem.  We also use some of the same methods used to find integer solutions $(x_1,\ldots,x_s)$ to the system
\begin{equation}
x_1^j+\ldots+x_s^j=N_j\qquad(1\leq j\leq k),
\label{equation38}
\end{equation}
where $N_1,\ldots,N_k$ are integers satisfying certain size conditions.  The problem of establishing a lower bound on $s$ in terms of $k$ such that the system (\ref{equation38}) possesses solutions in $\mathbb{N}^s$ for $N_1,\ldots,N_k$ satisfying appropriate conditions is known as the \textit{Hilbert-Kamke problem}.

Our proof broadly combines the Hardy-Littlewood circle method approach to counting three-term arithmetic progressions as described in \cite{roth1}, and the approach of Gowers to the general case of Szemer\'{e}di's theorem in \cite{gowers2}.  We first assume that the set $\mathcal{A}_N$ is $\mathfrak{a}$-uniform of degree $k$ according to the definition in \cite{gowers2} for $\mathfrak{a}$ appropriately bounded in terms of $\delta_N$.  In $\S$2, we approximate the number of solutions to (\ref{equation1}) in $\mathcal{A}_N^s$ by considering instead the number of solutions in $[1,N]^s$.  We use the methods developed by Gowers in \cite{gowers2} to estimate the error involved in this approximation in $\S$3.  After introducing some definitions and technical results from the Hardy-Littlewood circle method and Vinogradov's mean value theorem in $\S$4, we apply Vinogradov's result directly to complete our error estimate in $\S$5, while in $\S$6 we apply the translation and dilation invariance of the system (\ref{equation1}) and the Hardy-Littlewood method to count the number of solutions in $[1,N]^s$.  We then combine the resulting estimates in $\S$7 to obtain a lower bound for $\mathcal{N}$ in the uniform case for $\mathfrak{a}$ suitably bounded above in terms of $\delta_N$.  If the set $\mathcal{A}_N$ fails to be $\mathfrak{a}$-uniform of degree $k$ for $\mathfrak{a}$ suitably bounded above, we show the set is concentrated in a long arithmetic progression.  By iterating the concentration process, we obtain the upper bound for $\delta_N$ given in Theorem \ref{thm1}.

In this paper, $\ll$ and $\gg$ denote the familiar Vinogradov notation, and we use the familiar shorthand $e(x)$ for the exponential function $\exp(2\pi ix)$ and $e_q(x)$ for $\exp(2\pi ix/q)$, while $\mathbb{T}$ denotes the unit interval $[0,1)$.  Unless otherwise indicated, a boldface character denotes a $k$-dimensional vector, such as $\bfalpha$ for $(\alpha_k,\ldots,\alpha_1)$.  All statements involving the variable $\epsilon$ are assumed to hold for all values of $\epsilon>0$.

\vspace{4mm}

\section{Setup and notation}
Our basic approach to counting solutions to the system (\ref{equation1}) with $\bfx\in\mathcal{A}_N^s$ is by means of the Hardy-Littlewood circle method.  With this in mind, we define
\begin{displaymath}
f(\bfalpha)=\sum_{x\in\mathcal{A}_N}e(\alpha_k x^k+\ldots+\alpha_1 x),\qquad f_i(\bfalpha)=f(\lambda_i\bfalpha).
\end{displaymath}
It will often be convenient to write simply $f_i$ for $f_i(\bfalpha)$.  By orthogonality, the number $\mathcal{N}$ of solutions to (\ref{equation1}) in $\mathcal{A}_N^s$ may be written as
\begin{equation}
\mathcal{N}=\int_{\mathbb{T}^k}\prod_{i=1}^s f(\lambda_i\bfalpha)\, d\bfalpha=\int_{\mathbb{T}^k}\prod_{i=1}^s f_i\, d\bfalpha.
\label{equation3}
\end{equation}

Our aim is to establish a lower bound for $\mathcal{N}$ of the form $\mathcal{N}\gg\delta_N^s N^{s-k(k+1)/2}$, where $\delta_N$ is the density of $\mathcal{A}_N$.  In order to obtain an asymptotic approximation to the integral in (\ref{equation3}) with the expected main term, we require the definitions
\begin{displaymath}
\begin{array}{rclrcl}
g(\bfalpha)&=&\displaystyle{\sum_{1\leq x\leq N}e(\alpha_k x^k+\ldots+\alpha_1 x)}, & g_i(\bfalpha)&=&g(\lambda_i\bfalpha),\vspace{1mm}\\
v(\bfalpha)&=&\delta_N g(\bfalpha), & v_i(\bfalpha)&=&\delta_N g_i(\bfalpha),\vspace{1mm}\\
E(\bfalpha)&=&v(\bfalpha)-f(\bfalpha), & E_i(\bfalpha)&=&v_i(\bfalpha)-f_i(\bfalpha).
\end{array}
\end{displaymath}
As in the case of $f_i$, it will usually be convenient to omit the input variables for functions $g_i$, $v_i$, and $E_i$ in subsequent discussion.

To estimate the integral in (\ref{equation3}), we approximate the generating functions $f_i$ with the exponential sums $v_i$ and estimate both the new integral and the error involved in this approximation.  If we define
\begin{equation}
I=\int_{\mathbb{T}^k}\prod_{i=1}^s g_i\, d\bfalpha,
\label{equation30}
\end{equation}
then it follows by the triangle inequality that
\begin{displaymath}
\left|\delta_N^s\cdot I-\mathcal{N}\right|\ll\displaystyle{\int_{\mathbb{T}^k}\left|\prod_{i=1}^s v_i-\prod_{i=1}^s f_i\right| \, d\bfalpha}\ll \sum_{j=1}^s I_j,
\end{displaymath}
where
\begin{displaymath}
\begin{array}{lll}
I_1&=& \displaystyle{\int_{\mathbb{T}^k}\left|\left(v_1-f_1\right)\cdot\prod_{i=2}^s v_i\right| \, d\bfalpha},\vspace{1mm}\\
I_j &=& \displaystyle{\int_{\mathbb{T}^k}\left|\left(v_j-f_j\right)\cdot\prod_{i=1}^{j-1} v_i\prod_{i=j+1}^s f_i\right| \, d\bfalpha}\qquad (2\leq j\leq s-1),\vspace{1mm}\\
I_s &=& \displaystyle{\int_{\mathbb{T}^k}\left|\left(v_s-f_s\right)\cdot\prod_{i=1}^{s-1} f_i\right| \, d\bfalpha}.
\end{array}
\end{displaymath}

It is convenient at this point to introduce the notation
\begin{displaymath}
J(\Theta)=\int_{\mathbb{T}^k}\left|\Theta(\bfalpha)\right|^{s-1}\, d\bfalpha
\end{displaymath}
for the $(s-1)^{st}$ moment integral of a function $\Theta$.  Now, by the trivial inequality $|a_1\cdots a_t|\ll |a_1|^t+\ldots+|a_t|^t$, we obtain
\begin{displaymath}
I_1\ll\sup_{\bfalpha\in\mathbb{T}^k}\left|E_1\right|\cdot\sum_{i=2}^s \int_{\mathbb{T}^k}\left|v_i\right|^{s-1} \, d\bfalpha\ll\sup_{\bfalpha\in\mathbb{T}^k}\left|E(\bfalpha)\right|\cdot J(v),
\end{displaymath}
where the implied constant depends on the $\lambda_i$.  Similarly,
\begin{displaymath}
I_s\ll\sup_{\bfalpha\in\mathbb{T}^k}\left|E_s\right|\cdot\sum_{i=1}^{s-1} \int_{\mathbb{T}^k}\left|f_i\right|^{s-1} \, d\bfalpha\ll\sup_{\bfalpha\in\mathbb{T}^k}\left|E(\bfalpha)\right|\cdot J(f).
\end{displaymath}
Meanwhile, for fixed $2\leq j\leq s-1$, we have
\begin{eqnarray*}
I_k&\ll&\displaystyle{\sup_{\bfalpha\in\mathbb{T}^k}\left|E_k\right|\int_{\mathbb{T}^k}\sum_{i=1}^{j-1}\left|v_i\right|^{s-1}+\sum_{i=j+1}^s\left|f_i\right|^{s-1} \, d\bfalpha}\vspace{1mm}\\
&\ll&\displaystyle{\sup_{\bfalpha\in\mathbb{T}^k}\left|E(\bfalpha)\right|\cdot\left(J(v)+J(f)\right)}.
\end{eqnarray*}
We therefore infer that
\begin{equation}
\left|\delta_N^s\cdot I-\mathcal{N}\right|\ll\left(\sup_{\bfalpha\in\mathbb{T}^k}|E(\bfalpha)|\right)\cdot\left(J(v)+J(f)\right).
\label{equation16}
\end{equation}

Now, by the trivial estimates,
\begin{displaymath}
J(v)\ll\left(\delta_N N\right)^{s-1-s_0(k)}\int_{\mathbb{T}^k}\left|v(\bfalpha)\right|^{s_0(k)} \, d\bfalpha
\end{displaymath}
and
\begin{displaymath}
J(f)\ll\left(\delta_N N\right)^{s-1-s_0(k)}\int_{\mathbb{T}^k}\left|f(\bfalpha)\right|^{s_0(k)} \, d\bfalpha,
\end{displaymath}
where $s_0(k)$ is as defined in (\ref{equation37}).  By the underlying Diophantine equations, we observe that
\begin{displaymath}
\int_{\mathbb{T}^k}\left|f(\bfalpha)\right|^{s_0(k)} \, d\bfalpha\ll\int_{\mathbb{T}^k}\left|g(\bfalpha)\right|^{s_0(k)} \, d\bfalpha.
\end{displaymath}
It follows that if we define
\begin{displaymath}
J=\int_{\mathbb{T}^k} \left|g(\bfalpha)\right|^{s_0(k)} \, d\bfalpha,
\end{displaymath}
then by the preceding discussion and the definition of $v(\bfalpha)$, we have
\begin{displaymath}
J(v)+J(f)\ll\left(\delta_N N\right)^{s-1-s_0(k)}\cdot J.
\end{displaymath}
Substituting this expression into (\ref{equation16}) yields
\begin{equation}
\left|\delta_N^s\cdot I-\mathcal{N}\right|\ll\left(\sup_{\bfalpha\in\mathbb{T}^k}|E(\bfalpha)|\right)\cdot\left(\delta_N N\right)^{s-1-s_0(k)}\cdot J.
\label{equation11}
\end{equation}

To obtain a lower bound for $\mathcal{N}$, we therefore require an upper bound for the error $E(\bfalpha)$, an upper bound for the moment integral $J$, and a lower bound for the approximation integral $I$.

\vspace{4mm}

\section{The error $E(\bfalpha)$ via uniformity of degree $k$}
We begin our analysis of the error $E(\bfalpha)$ by assuming that the set $\mathcal{A}_N$ is $\mathfrak{a}$-uniform of degree $k$ according to the definition in \cite{gowers2} for $\mathfrak{a}$ suitably bounded in terms of $\delta_N$.  We then use the method of Weyl differencing to obtain an exponential sum which may be estimated using the method developed in \cite{gowers2}.

\begin{lemma}
If the set $\mathcal{A}_N$ is $\mathfrak{a}$-uniform of degree $k$, then
\begin{equation}
|E(\bfalpha)|\leq 2\mathfrak{a}^{2^{-k-1}}N
\label{equation10}
\end{equation}
uniformly in $\bfalpha$.
\label{lem9}
\end{lemma}

\begin{proof}
We begin by defining $A(x)$ to be the characteristic function of $\mathcal{A}_N$ and
\begin{displaymath}
\mathcal{E}_N(x)=\left\{\begin{array}{ll} \delta_N-A(x), & \textrm{when }1\leq x\leq N,\\ 0, & \textrm{otherwise,}\end{array}\right.
\end{displaymath}
the \textit{balanced function} of the set $\mathcal{A}_N$.  This allows us to write
\begin{displaymath}
E(\bfalpha)=\sum_{x=1}^N \mathcal{E}_N(x)e(\alpha_k x^k+\ldots+\alpha_1 x).
\end{displaymath}

To estimate this sum for general $\bfalpha$, we use the method of Weyl differencing (see, for example, $\S$2.2 in \cite{vaughan}).  It is convenient at this point to introduce the following shorthand notation for the associated difference operators.  For any function $f$, we define the first forward difference operators $\Delta_1$ by
\begin{displaymath}
\Delta_1(f(x); w)=f(x)\overline{f(x-w)}.
\end{displaymath}
For higher order forward difference operators, we define $\Delta_k$ recursively by
\begin{displaymath}
\Delta_k(f(x); w_1,\ldots,w_k)=\Delta_1\left(\Delta_{k-1}(f(x);w_1,\ldots,w_{k-1}),w_k\right).
\end{displaymath}
Under this notation, we recall from $\S$3 of \cite{gowers2} that a set $\mathcal{A}_N\subset [1,N]$ of cardinality $\delta_N N$ is said to be \textit{$\mathfrak{a}$-uniform of degree $k$} for a parameter $\mathfrak{a}$ if
\begin{equation}
\sum_{\left|w_1\right|,\ldots,\left|w_{k+1}\right|\leq N-1}\sum_{x\in I_{\bfw}}\Delta_{k+1}\left(\mathcal{E}_N(x);\bfw\right)\leq\mathfrak{a}N^{k+2},
\label{equation7}
\end{equation}
where the interval $I_{\bfw}$, which may be empty, is defined by
\begin{equation}
\begin{array}{lll}
I_{\bfw}&=&[1,N]\cap\left([1,N]+w_1\right)\cap\ldots\cap\left([1,N]+w_1+\ldots+w_{k+1}\right).
\end{array}
\label{equation18}
\end{equation}
Here $[1,N]+w$ denotes the right translation of the interval $[1,N]$ by $w$.  The statement of this definition differs slightly from that in \cite{gowers2}, in which the set $\mathcal{A}_N$ is taken to be a subset of $\mathbb{Z}/N\mathbb{Z}$ rather than $\mathbb{Z}$, and so the sum is taken instead over $x,w_1,\ldots,w_{k+1}\in\mathbb{Z}/N\mathbb{Z}$.  However, an elementary computation reveals that the definition above is equivalent to that in \cite{gowers2}, as the definition of $I_{\bfw}$ ensures that the quantities
\begin{displaymath}
x-\epsilon_1 w_1-\ldots-\epsilon_{k+1} w_{k+1}\qquad\left(\epsilon_1,\ldots,\epsilon_{k+1}\right)\in\{0,1\}^{k+1}
\end{displaymath}
all lie in the interval $[1,N]$ for a given $(k+1)$-tuple $\bfw$.

If we now apply $k+1$ iterations of Weyl differencing to the sum $E(\bfalpha)$, we obtain
\begin{displaymath}
\left|E(\bfalpha)\right|^{2^{k+1}} \leq (2N)^{2^{k+1}-k-2} \sum_{\left|w_1\right|,\ldots,\left|w_{k+1}\right|\leq N-1}\sum_{x\in I_{\bfw}}\Delta_{k+1}\left(\mathcal{E}_N(x);\bfw\right),
\end{displaymath}
where we retain the definition of $I_{\bfw}$ from (\ref{equation18}).  We note that after $k+1$ iterations of Weyl differencing, the exponential factor $e(\alpha_k x^k+\ldots+\alpha_1 x)$ is reduced to $1$.  Now, if the set $\mathcal{A}_N$ is $\mathfrak{a}$-uniform of degree $k$, then it follows by (\ref{equation7}) that
\begin{displaymath}
\left|E(\bfalpha)\right|^{2^{k+1}}\leq (2N)^{2^{k+1}-k-2}\cdot\mathfrak{a}N^{k+2}=2^{2^{k+1}-k-2}\mathfrak{a} N^{2^{k+1}}.
\end{displaymath}
By raising both sides to the power $2^{-k-1}$, we obtain
\begin{displaymath}
\left|E(\bfalpha)\right|<2\mathfrak{a}^{2^{-k-1}} N,
\end{displaymath}
which is precisely the bound stated in (\ref{equation10}).
\end{proof}

\vspace{4mm}

\section{Hardy-Littlewood method preliminaries}
To estimate both the moment integral $J$ and the approximation integral $I$ in (\ref{equation11}), we apply the classical Hardy-Littlewood circle method by subdividing the $k$-dimensional unit cube $\mathbb{T}^k$ into the appropriately defined major arcs $\mathfrak{M}$ and the corresponding minor arcs $\mathfrak{m}$ and then estimating the contribution of the integrals over each set.  To obtain these estimates, we require several definitions and technical lemmas.

We first define the dissection of $\mathbb{T}^k$ into the major and minor arcs.  Let
\begin{equation}
\sigma(k)^{-1}=8k^2\left(\log k+(\log\log k)/2+2\right),\qquad\delta(k)=k\sigma(k).
\label{equation36}
\end{equation}
For $q\in\mathbb{N}$, $\bfa\in\mathbb{Z}^k$, we may define an individual major arc $\mathfrak{M}(q,\bfa)$ by
\begin{displaymath}
\mathfrak{M}(q,\bfa)=\left\{\bfalpha\in[0,1)^k:\left|q\alpha_j-a_j\right|\leq N^{\delta(k)-j}, 1\leq j\leq k \right\}.
\end{displaymath}
This allows us to define the major arcs $\mathfrak{M}$ and the minor arcs $\mathfrak{m}$ to be
\begin{equation}
\mathfrak{M}=\bigcup_{\substack{0\leq a_k,\ldots,a_1\leq q\leq N^{\delta(k)}\\(q,a_k,\ldots,a_1)=1}} \mathfrak{M}(q,\bfa),\qquad \mathfrak{m}=[0,1)^k\backslash\mathfrak{M}.
\label{equation15}
\end{equation}
It follows that the integral of any function over $\mathbb{T}^k$ is the sum of the integrals over $\mathfrak{M}$ and $\mathfrak{m}$.

Our estimates of the minor arc contributions require the following lemma.

\begin{lemma}
Let $\sigma(k)$ and $\delta(k)$ be as defined in (\ref{equation36}).  Suppose that $N$ is sufficiently large in terms of $k$ and that $|g(\bfalpha)|\geq N^{1-\sigma(k)}$.  Then there exist $q\in\mathbb{N}$ and $\bfa\in\mathbb{Z}^k$ such that $1\leq q\leq N^{\delta(k)}$ and $|q\alpha_j-a_j|\leq N^{\delta(k)-j}$ for $1\leq j\leq k$.
\label{lem4}
\end{lemma}
\begin{proof}
This is a special case of Theorem 4.4 in \cite{baker}.
\end{proof}

We note that if we define $Q=(q,a_k,\ldots,a_1)$ and $q^*=q/Q$, $a_j^*=a_j/Q$ for $1\leq j\leq k$, then it follows by Lemma \ref{lem4} that $|g(\bfalpha)|\geq N^{1-\sigma(k)}$ necessarily forces $\bfalpha\in\mathfrak{M}$.  The contrapositive of this result implies that $\sup_{\bfalpha\in\mathfrak{m}}|g(\bfalpha)|\ll N^{1-\sigma(k)}$.

To estimate the contribution of integrals over the major arcs $\mathfrak{M}$, we require the definitions
\begin{displaymath}
\begin{array}{rclrcl}
S(q,\bfa)&=&\displaystyle{\sum_{m=1}^q e_q\left(a_k m^k+\ldots+a_1 m\right)}, & S_i(q,\bfa)&=&\displaystyle{S(q,\lambda_i\bfa)},\vspace{1mm}\\
w(\bfbeta)&=&\displaystyle{\int_0^N e\left(\beta_k\gamma^k+\ldots+\beta_1\gamma\right)\, d\gamma,} & w_i(\bfbeta)&=&\displaystyle{w(\lambda_i\bfbeta).}
\end{array}
\end{displaymath}
We also require the following lemma, which we state in the generality necessary for use in estimating both the moment integral $J$ and the approximation integral $I$.

\begin{lemma}
Suppose that $(q,a_k,\ldots,a_1)=1$.  Then the series $S_i(q,\bfa)$ and the integral $w_i(\bfbeta)$ satisfy
\begin{equation}
S_i(q,\bfa)\ll q^{1-1/k+\epsilon}
\label{equation5}
\end{equation}
and
\begin{equation}
w_i(\bfbeta)\ll N(1+|\beta_1| N+\ldots+|\beta_k| N^k)^{-1/k}.
\label{equation8}
\end{equation}
Moreover, if $\bfalpha\in\mathfrak{M}$ and we define $\beta_j=\alpha_j-a_j/q$ for $1\leq j\leq k$, then
\begin{equation}
g_i(\bfalpha)-q^{-1} S_i(q,\bfa) w_i(\bfbeta)\ll q(1+|\beta_1|N+\ldots+|\beta_k|N^k).
\label{equation9}
\end{equation}
The implied constants in each case depend on $\lambda_i$.
\label{lem3}
\end{lemma}
\begin{proof}
These results all follow from the discussion in Chapter 7 of \cite{vaughan}.  In particular, the bounds in (\ref{equation5}) and (\ref{equation8}) follow respectively from Theorem 7.1 and Theorem 7.3, while the error estimate in (\ref{equation9}) follows from Theorem 7.2.
\end{proof}

To estimate the main term in the major arc contribution for the approximation integral $I$, we require the following lemma on the sums $S_i(q,\bfa)$.

\begin{lemma}
Suppose $(q,a_k,\ldots,a_1)=(r,b_k,\ldots,b_1)=(q,r)=1$.  Then
\begin{displaymath}
S_i(qr,a_k r+b_k q,\ldots,a_1 r+b_1 q)=S_i(q,a_k,\ldots,a_1)S_i(r,b_k,\ldots,b_1).
\end{displaymath}
Furthermore, the function $S(q)$ defined by
\begin{equation}
S(q)=\sum_{\substack{0\leq a_k,\ldots,a_1\leq q\\(q,a_k,\ldots,a_1)=1}}q^{-s}\prod_{i=1}^s S_i(q,\bfa)
\label{equation13}
\end{equation}
is multiplicative.
\label{lem12}
\end{lemma}
\begin{proof}
The first result is essentially a generalisation of Lemma 2.10 in \cite{vaughan}, while the second result follows from the first and from Lemma 2.11 in \cite{vaughan}.
\end{proof}

The later steps of our analysis of the major arc contribution to our estimate for the approximation integral $I$ rest on establishing the existence of non-singular real solutions to the system (\ref{equation1}) in $(0,1)^s$ and non-singular $p$-adic solutions for all rational primes $p$.

\begin{lemma}
The system (\ref{equation1}) possesses a real, non-singular solution in $(0,1)^s$.
\label{lem7}
\end{lemma}
\begin{proof}
That the system (\ref{equation1}) possesses a real, non-singular solution $\bfy=(y_1,\ldots,y_s)$ is a hypothesis of Theorems \ref{thm1} and \ref{thm2}.  If $\bfy\notin(0,1)^s$, then we may generate a real, non-singular solution in $(0,1)^s$ as follows.  Define
\begin{displaymath}
Y=\max_{1\leq i\leq s} \left|y_i\right|.
\end{displaymath}
Then the vector $(\eta_1,\ldots,\eta_s)$, where
\begin{displaymath}
\eta_i=\frac{y_i}{4Y}+\frac{1}{2}\qquad(1\leq i\leq s),
\end{displaymath}
lies in the $s$-dimensional cube $(0,1)^s$ and is also a non-singular solution to the system (\ref{equation1}) by translation and dilation invariance.
\end{proof}

Before considering the existence of $p$-adic solutions, we let $q$ be any natural number and define $M_n(q)$ to be the number of solutions to the system of congruences
\begin{equation}
\lambda_1 x_1^j+\ldots+\lambda_s x_s^j \equiv 0 \pmod q\qquad (1\leq j\leq k)
\label{equation27}
\end{equation}
with $\bfx\in(\mathbb{Z}/q\mathbb{Z})^s$.  We also require the following version of Hensel's Lemma, which is essentially Proposition 5.20 in \cite{greenberg}.

\begin{lemma}(Hensel's Lemma)
Suppose that $L_1(X_1,\ldots,X_k),\ldots, L_k(X_1,\ldots,X_k)\in\mathbb{Z}_p[X_1,\ldots,X_k]$ and $(x_1,\ldots,x_k)\in\mathbb{Z}_p^k$.  Define
\begin{displaymath}
\Delta_0=\left.\det\left(\frac{\partial L_j}{\partial X_i}\right)\right|_{(x_1,\ldots,x_k)},
\end{displaymath}
where the partial derivatives are formal derivatives.  Suppose that $\Delta_0\neq 0$ and that
\begin{displaymath}
\max_{1\leq j\leq k}\left\{\left|L_j(x_1,\ldots,x_k)\right|_p\right\}<\left|\Delta_0\right|_p^2.
\end{displaymath}
Then there exists a unique vector $\bfy=(y_1,\ldots,y_k)\in\mathbb{Z}_p^k$ such that $L_1(\bfy)=\ldots=L_k(\bfy)=0$ and
\begin{displaymath}
\max_{1\leq j\leq k}\left\{\left|x_j-y_j\right|_p\right\}\leq p^{-1}\cdot\left|\Delta_0\right|_p.
\end{displaymath}
\label{lem1}
\end{lemma}

\begin{lemma}
For every rational prime $p$ there exists a number $u=u(p)<\infty$ such that
\begin{displaymath}
M_n(p^t)\geq p^{(t-u)(s-k)}
\end{displaymath}
for all $t\geq u$.
\label{lem11}
\end{lemma}
\begin{proof}
Suppose that $(a_1,\ldots,a_s)$ is a non-singular solution in $\mathbb{Z}_p^s$ to the system of congruences (\ref{equation27}).  That such a solution exists is a hypothesis of Theorems \ref{thm1} and \ref{thm2}.  By suitable re-numbering, we may assume that $\Delta(1,\ldots,k)\neq 0$, where $\Delta(1,\ldots,k)$ is the Jacobian determinant defined in (\ref{equation40}).

For an $(s-k)$-dimensional vector $\bfz=(z_{k+1},\ldots,z_s)$, we define
\begin{displaymath}
\Lambda_j(\bfz)=\lambda_{k+1} z_{k+1}^j+\ldots+\lambda_s z_s^j\qquad(1\leq j\leq k).
\end{displaymath}
Suppose $\left|\Delta(1,\ldots,k)\right|_p^2=p^{1-u}>0$.  We choose $z_{k+1},\ldots,z_s$ so that
\begin{equation}
z_i\equiv a_i\pmod {p^t}\qquad(t\geq u).
\label{equation29}
\end{equation}
Then
\begin{displaymath}
\lambda_1 a_1^j+\ldots+\lambda_k a_k^j+\Lambda_j(\bfz) \equiv  0 \pmod {p^u} \qquad  (1\leq j\leq k).
\end{displaymath}
It follows that
\begin{displaymath}
\max_{1\leq j\leq k}\left\{\left|\lambda_1 a_1^j+\ldots+\lambda_k a_k^j+\Lambda_j(\bfz)\right|_p\right\}\leq p^{-u}<\left|\Delta(1,\ldots,k)\right|_p^2,
\end{displaymath}
and by Lemma \ref{lem1} there exists a unique $(b_1,\ldots,b_k)\in\mathbb{Z}_p^k$ such that
\begin{displaymath}
\lambda_1 b_1^j+\ldots+\lambda_k b_k^j+\Lambda_j(\bfz) = 0 \qquad  (1\leq j\leq k).
\end{displaymath}
And since by (\ref{equation29}) we have $p^{t-u}$ choices for each of $z_{k+1},\ldots,z_s$, it follows that there are at least $p^{(t-u)(s-k)}$ solutions to the congruences in (\ref{equation27}) with $q=p^t$.  Lemma \ref{lem11} follows immediately.
\end{proof}

\vspace{4mm}

\section{The moment integral $J$}
As the moment integral $J$ is identical to the integral at the heart of the Vinogradov mean value theorem for degree $k$, an upper bound for $J$ follows almost immediately from Vinogradov's result.  We need only establish that the associated singular integral $\mathcal{J}(J)$ and singular series $\mathfrak{S}(J)$ are absolutely convergent.

\begin{lemma}
The moment integral $J$ satisfies the bound $J\ll N^{s_0(k)-k(k+1)/2}.$
\label{lem8}
\end{lemma}
\begin{proof}
We first define the singular series and singular integral for $J$ respectively by
\begin{displaymath}
\begin{array}{c}
\displaystyle{\mathfrak{S}(J)=\sum_{q=1}^{\infty}\sum_{\substack{0\leq a_k,\ldots,a_1\leq q\\(q,a_k,\ldots,a_1)=1}}q^{-s_0(k)}\left|S(q,\bfa)\right|^{s_0(k)}},\vspace{1mm}\\
\displaystyle{\mathcal{J}(J)=\int_{\mathbb{R}^k}\left|w(\bfbeta)\right|^{s_0(k)} \, d\bfbeta}.
\end{array}
\end{displaymath}
The proof of Theorem 3 in \cite{wooley} states that
\begin{displaymath}
\int_{\mathbb{T}^k}|g(\bfalpha)|^{s_0(k)}\, d\bfalpha - \mathfrak{S}(J)\mathcal{J}(J)\ll N^{s_0(k)-k(k+1)/2-\tau}
\end{displaymath}
for some $\tau>0$.  We therefore need only prove that $\mathfrak{S}(J)\mathcal{J}(J)\ll N^{s_0(k)-k(k+1)/2}$.

By (\ref{equation8}), the singular integral $\mathcal{J}(J)$ obeys the bound
\begin{displaymath}
\begin{array}{rcl}
\mathcal{J}(J) &\ll & \displaystyle{N^{s_0(k)} \int_{\mathbb{R}^k}(1+|\beta_1|N+\ldots|\beta_k|N^k)^{-s_0(k)/k}\, d\bfbeta}\vspace{1mm}\\
&\ll & \displaystyle{N^{s_0(k)}\prod_{j=1}^k\int_{-\infty}^{\infty}(1+|\beta_j|N^j)^{-s_0(k)/k^2}\, d\beta_j}\vspace{1mm}\\
&\ll & \displaystyle{N^{s_0(k)}\prod_{j=1}^k N^{-j} \,\,=\,\, N^{s_0(k)-k(k+1)/2}}.
\end{array}
\end{displaymath}
It follows that $\mathcal{J}(J)$ converges absolutely.  Meanwhile, by (\ref{equation5}), we have
\begin{displaymath}
\begin{array}{rcl}
\mathfrak{S}(J) &=& \displaystyle{\sum_{q=1}^{\infty}\sum_{\substack{0\leq a_k,\ldots,a_1\leq q\\(q,a_k,\ldots,a_1)=1}}q^{-s_0(k)}\left|S(q,\bfa)\right|^{s_0(k)}}\vspace{1mm}\\
&\ll & \displaystyle{\sum_{q=1}^{\infty} q^k\cdot q^{s_0(k)\cdot(-1/k+\epsilon)}\,\,\ll\,\, 1.}
\end{array}
\end{displaymath}
This establishes the absolute convergence of $\mathfrak{S}(J)$.  We have therefore shown that
\begin{displaymath}
J\ll\mathfrak{S}(J)\mathcal{J}(J)\ll N^{s_0(k)-k(k+1)/2},
\end{displaymath}
which is precisely the bound in Lemma \ref{lem8}.
\end{proof}

\vspace{4mm}

\section{The approximation integral $I$}
Our estimate the approximation integral $I$ is obtained by means of the Hardy-Littlewood circle method.  We dissect the $k$-dimensional unit cube $\mathbb{T}^k$ into the major and minor arcs and then estimate the contribution of each region in turn.

We first define the singular series and singular integral for $I$ respectively by
\begin{displaymath}
\begin{array}{c}
\displaystyle{\mathfrak{S}(I)=\sum_{q=1}^{\infty}\sum_{\substack{0\leq a_k,\ldots,a_1\leq q\\(q,a_k,\ldots,a_1)=1}}q^{-s}\prod_{i=1}^s S_i(q,\bfa)},\vspace{1mm}\\
\displaystyle{\mathcal{J}(I)=\int_{\mathbb{R}^k}\prod_{i=1}^s w_i(\bfbeta) \, d\bfbeta},
\end{array}
\end{displaymath}
where $S_i(q,\bfa)$ and $w_i(\bfbeta)$ are as defined in $\S$4.4.  We establish in the next lemma that the error involved in approximating $I$ by the product $\mathfrak{S}(I)\mathcal{J}(I)$ is of smaller magnitude than the expected main term.

\begin{lemma}
For $s>s_0(k)$, we have
\begin{displaymath}
I - \mathfrak{S}(I)\mathcal{J}(I) \ll N^{s-k(k+1)/2-\Delta}
\end{displaymath}
for some $\Delta>0$.
\label{lem10}
\end{lemma}

\begin{proof}
We retain the definitions of $\mathfrak{M}$ and $\mathfrak{m}$ from (\ref{equation15}).  Define 
\begin{displaymath}
\Lambda=\max_{1\leq i\leq s}\left|\lambda_i\right|.
\end{displaymath}
By suitable re-numbering, we may assume that
\begin{displaymath}
\max_{1\leq i\leq s}\left(\sup_{\bfalpha\in\mathfrak{m}}\left|g_i(\bfalpha)\right|\right)=\sup_{\bfalpha\in\mathfrak{m}}\left|g_1(\bfalpha)\right|.
\end{displaymath}
If we apply H\"{o}lder's inequality, we see that the integral over the minor arcs $\mathfrak{m}$ is bounded by
\begin{displaymath}
\int_{\mathfrak{m}}\prod_{i=1}^s |g_i| \, d\bfalpha\ll\left(\sup_{\bfalpha\in\mathfrak{m}}\left|g_1(\bfalpha)\right|\right)\prod_{2\leq i\leq s} \left(\int_{\mathbb{T}^k}|g_i|^{s-1} \, d\bfalpha\right)^{1/(s-1)}.
\end{displaymath}

Now, by replacing $\sigma(k)$ with $\sigma(k)/2$ in the statement of Lemma \ref{lem4}, we see that if $|g_1(\bfalpha)|\geq N^{1-\sigma(k)/2}$, then there exist $q\in\mathbb{N}$, $q<N^{\delta(k)/2}$ and $\bfa\in\mathbb{Z}^k$ such that $(q,a_k,\ldots,a_1)=1$ and $|q\lambda_1\alpha_j-a_j|<N^{\delta(k)/2-j}$ for $1\leq j\leq k$.  We now define
\begin{displaymath}
Q=(q\lambda_1,a_k,\ldots,a_1)
\end{displaymath}
and let $\sgn(\lambda_1)=\lambda_1/|\lambda_1|$ denote the signum function of $\lambda_1$.  We further define
\begin{displaymath}
q^{\prime}=q\left|\lambda_1\right|/Q,\qquad a_j^{\prime}=\sgn\left(\lambda_1\right)\cdot a_j/Q \quad(1\leq j\leq k).
\end{displaymath}
If we assume that $N$ is sufficiently large that $\Lambda<N^{\delta(k)/2}$, then we have $q^{\prime}<N^{\delta(k)}$, $0\leq a_k^{\prime},\ldots,a_1^{\prime}<q^{\prime}$, $(q^{\prime},a_k^{\prime},\ldots,a_1^{\prime})=1$, and $|q^{\prime}\alpha_j-a_j^{\prime}|<N^{\delta(k)-j}$ for $1\leq j\leq k$.  It follows by (\ref{equation15}) that $\bfalpha\in\mathfrak{M}$, and therefore,
\begin{displaymath}
\sup_{\bfalpha\in\mathfrak{m}}\left|g_1(\bfalpha)\right|\ll N^{1-\sigma(k)/2}.
\end{displaymath}
Meanwhile, since $s-1\geq s_0(k)$, we have by Lemma \ref{lem8} and the trivial estimate $|g(\bfalpha)|\leq N$ that
\begin{displaymath}
\prod_{2\leq i\leq s} \left(\int_{\mathbb{T}^k}|g_i|^{s-1} \, d\bfalpha\right)^{1/(s-1)} \ll \int_{\mathbb{T}^k}|g(\bfalpha)|^{s-1}\, d\bfalpha \ll N^{s-1-k(k+1)/2},
\end{displaymath}
where the implied constants depend on the $\lambda_i$.  Hence,
\begin{equation}
\int_{\mathfrak{m}}\prod_{i=1}^s |g_i| \, d\bfalpha\ll N^{s-k(k+1)/2-\tau}
\label{equation23}
\end{equation}
for some $\tau>0$.

On the major arcs $\mathfrak{M}$, we retain our previous definitions of $S_i(q,\bfa)$ and $w_i(\bfbeta)$, where $\beta_j=\alpha_j-a_j/q$ for $1\leq j\leq k$.  It therefore follows by (\ref{equation9}) that
\begin{displaymath}
\begin{array}{ll}
\lefteqn{\displaystyle{\int_{\mathfrak{M}}\left|\prod_{i=1}^s g_i-q^{-s}\prod_{i=1}^s S_i(q,\bfa)w_i(\bfbeta)\right|\, d\bfalpha}}\vspace{1mm}\\
&\ll\,\, \displaystyle{\int_{\mathfrak{M}}q(1+|\beta_1|N+\ldots+|\beta_k|N^k)N^{s-1} \, d\bfalpha\,\,\ll\,\, N^{s-1+\delta(k)}}\cdot mes(\mathfrak{M}).
\end{array}
\end{displaymath}
And since
\begin{displaymath}
\begin{array}{lll}
mes(\mathfrak{M})&\ll&\displaystyle{\sum_{q=1}^{N^{\delta(k)}}\sum_{\substack{0\leq a_k,\ldots,a_1\leq q\\(q,a_k,\ldots,a_1)=1}} q^{-k} N^{k\delta(k)-k(k+1)/2}}\vspace{1mm}\\
&\ll&\displaystyle{\sum_{q=1}^{N^{\delta(k)}}N^{k\delta(k)-k(k+1)/2}\,\,\ll\,\, N^{(k+1)\delta(k)-k(k+1)/2}},
\end{array}
\end{displaymath}
we obtain the bound
\begin{equation}
\begin{array}{ll}
\lefteqn{\displaystyle{\int_{\mathfrak{M}}\left|\prod_{i=1}^s g_i-q^{-s}\prod_{i=1}^s S_i(q,\bfa)w_i(\bfbeta)\right|\, d\bfalpha}}\vspace{1mm}\\
&\ll\,\, N^{s-1+\delta(k)}\cdot N^{(k+1)\delta(k)-k(k+1)/2}\,\,\ll\,\, N^{s-k(k+1)/2-\tau^{\prime}}
\end{array}
\label{equation24}
\end{equation}
for some $\tau^{\prime}>0$, where the implied constants depend on the $\lambda_i$.  That $\tau^{\prime}>0$ follows by observing that $(k+2)\delta(k)<(k+2)/8k<1$ for $k\geq 2$.

We now define the truncated singular series and truncated singular integral for $I$ respectively by
\begin{displaymath}
\begin{array}{c}
\displaystyle{\mathfrak{S}(I,N^{\delta(k)})=\sum_{q=1}^{N^{\delta(k)}}\sum_{\substack{0\leq a_k,\ldots,a_1\leq q\\(q,a_k,\ldots,a_1)=1}}q^{-s}\prod_{i=1}^s S_i(q,\bfa)},\vspace{1mm}\\
\displaystyle{\mathcal{J}(I,N^{\delta(k)},q)=\int_{-q^{-1} N^{\delta(k)-k}}^{q^{-1} N^{\delta(k)-k}}\cdots\int_{-q^{-1} N^{\delta(k)-1}}^{q^{-1} N^{\delta(k)-1}}\prod_{i=1}^s w_i(\bfbeta)\, d\bfbeta}.
\end{array}
\end{displaymath}

By the definition of the truncated singular integral, we have
\begin{equation}
\begin{array}{lll}
\lefteqn{\displaystyle{\int_{\mathfrak{M}} q^{-s}\prod_{i=1}^s S_i(q,\bfa)w_i(\bfbeta)\, d\bfalpha}}\vspace{1mm}\\
&=&\displaystyle{\sum_{q=1}^{N^{\delta(k)}}\sum_{\substack{0\leq a_k,\ldots,a_1\leq q\\(q,a_k,\ldots,a_1)=1}}\int_{-q^{-1} N^{\delta(k)-k}}^{q^{-1} N^{\delta(k)-k}}\cdots\int_{-q^{-1} N^{\delta(k)-1}}^{q^{-1} N^{\delta(k)-1}}q^{-s}\prod_{i=1}^s S_i(q,\bfa)w_i(\bfbeta)\, d\bfbeta}\vspace{1mm}\\
&=&\displaystyle{\sum_{q=1}^{N^{\delta(k)}}S(q)\mathcal{J}(I,N^{\delta(k)},q)},
\end{array}
\label{equation32}
\end{equation}
where $S(q)$ is as defined in (\ref{equation13}).

Now, by (\ref{equation8}),
\begin{equation}
\begin{array}{lll}
\mathcal{J}(I) &\ll& \displaystyle{N^s\int_{\mathbb{R}^k}(1+|\beta_1|N+\ldots+|\beta_k|N^k)^{-s/k}\, d\bfbeta}\vspace{1mm}\\
&\ll& \displaystyle{N^s\prod_{j=1}^k\int_{\mathbb{R}}(1+|\beta_j|N^j)^{-s/k^2}d\beta_j}\vspace{1mm}\\
&\ll& \displaystyle{N^s\prod_{j=1}^k N^{-j}}=N^{s-k(k+1)/2}.
\end{array}
\label{equation25}
\end{equation}
It follows that the singular integral $\mathcal{J}(I)$ converges absolutely.  Moreover,
\begin{equation}
\begin{array}{ll}
\lefteqn{|\mathcal{J}(I)-\mathcal{J}(I,N^{\delta(k)},q)|}\vspace{1mm}\\
&\ll\,\, N^s\displaystyle{\sum_{j=1}^k\int_{q^{-1} N^{\delta(k)-j}}^{\infty} (1+|\beta_j|N^j)^{-s/k^2}d\beta_j\cdot\prod_{\substack{1\leq i\leq k\\i\neq j}}\int_0^{\infty}(1+|\beta_i|N^i)^{-s/k^2}\, d\beta_i}\vspace{1mm}\\
&\ll\,\, N^s\displaystyle{\sum_{j=1}^k q N^{-j-\delta(k)}\prod_{\substack{1\leq i\leq k\\i\neq j}}N^{-i}}\,\,\ll\,\, q N^{s-k(k+1)/2-\delta(k)},
\end{array}
\label{equation31}
\end{equation}
where the implied constants depend on the $\lambda_i$.

It remains to analyse the singular series.  By (\ref{equation5}),
\begin{equation}
S(q)\ll\sum_{\substack{0\leq a_k,\ldots,a_1\leq q\\(q,a_k,\ldots,a_1)=1}}q^{-s}\cdot q^{s(1-1/k+\epsilon)}\ll q^k\cdot q^{-s/k+s\epsilon}\ll q^{-9k},
\label{equation14}
\end{equation}
provided that $s>s_0(k)>10k^2$ and $\epsilon$ is chosen to be sufficiently small.  Therefore, by the multiplicative nature of $S(q)$,
\begin{displaymath}
\sum_{q\leq N^{\delta(k)}}q|S(q)|\leq\prod_{p\leq N^{\delta(k)}}\left(1+\sum_{l=1}^{\infty}p^l\left|S(p^l)\right|\right)\ll\prod_{p\leq N^{\delta(k)}}(1+Cp^{-9k+1})
\end{displaymath}
for some constant $C$.  Hence,
\begin{displaymath}
\sum_{q\leq N^{\delta(k)}}q|S(q)|\ll 1,
\end{displaymath}
where the implied constant depends on the $\lambda_i$.  Combining this result with (\ref{equation31}) gives
\begin{equation}
\begin{array}{lll}
\displaystyle{\sum_{q\leq N^{\delta(k)}}S(q)\left(\mathcal{J}(I)-\mathcal{J}(I,N^{\delta(k)},q)\right)}&\ll&\displaystyle{\sum_{q\leq N^{\delta(k)}} |S(q)|\cdot qN^{s-k(k+1)/2-\delta(k)}}\vspace{1mm}\\
&\ll& N^{s-k(k+1)/2-\delta(k)}.
\end{array}
\label{equation33}
\end{equation}

Meanwhile, by (\ref{equation14}),
\begin{equation}
\left|\mathfrak{S}(I)-\mathfrak{S}(I,N^{\delta(k)})\right|=\sum_{q\geq N^{\delta(k)}}|S(q)|\ll\sum_{q\geq N^{\delta(k)}}q^{-9k}\ll N^{-\delta(k)},
\label{equation12}
\end{equation}
where the implied constant depends on the $\lambda_i$.  Moreover,
\begin{displaymath}
\mathfrak{S}(I)\ll\sum_{q=1}^{\infty} q^{-9k}\ll 1.
\end{displaymath}
We have therefore demonstrated the absolute convergence of the singular series $\mathfrak{S}(I)$.  By combining this bound with (\ref{equation25}), (\ref{equation33}), and (\ref{equation12}), we have established
\begin{displaymath}
\begin{array}{lll}
\lefteqn{\displaystyle{\sum_{q\leq N^{\delta(k)}}S(q)\mathcal{J}(I,N^{\delta(k)},q)-\mathfrak{S}(I)\mathcal{J}(I)}}\vspace{1mm}\\
&&\ll\,\,\displaystyle{\sum_{q\leq N^{\delta(k)}}|S(q)|\left|\mathcal{J}(I)-\mathcal{J}(I,N^{\delta(k)},q)\right|+\left|\mathcal{J}(I)\right|\left|\mathfrak{S}(I)-\mathfrak{S}(I,N^{\delta(k)})\right|}\vspace{1mm}\\
&&\ll\,\, N^{s-k(k+1)/2-\delta(k)}+N^{s-k(k+1)/2}\cdot N^{-\delta(k)}\,\,\ll\,\, N^{s-k(k+1)/2-\delta(k)}.
\end{array}
\end{displaymath}
Lemma \ref{lem10} follows by combining this result with (\ref{equation23}), (\ref{equation24}), and (\ref{equation32}).
\end{proof}

\begin{lemma}
For some constant $\mathcal{C}>0$ dependent at most on the $\lambda_i$, we have $\mathcal{J}(I)=\mathcal{C}N^{s-k(k+1)/2}$.
\label{lem14}
\end{lemma}
\begin{proof}
By the change of variable $N\gamma_i$ for $\gamma_i$, we obtain
\begin{displaymath}
w_i(\bfbeta)=N\int_0^1 e(\lambda_i\beta_k(N\gamma_i)^k+\ldots+\lambda_i\beta_1(N\gamma_i)) \, d\gamma_i\qquad(1\leq i\leq s),
\end{displaymath}
and by the changes of variable $\beta_j N^{-j}$ for $\beta_j$ $(1\leq j\leq k)$, we deduce that $\mathcal{J}(I)=\mathcal{C}(N)\cdot N^{s-k(k+1)/2}$, where
\begin{displaymath}
\mathcal{C}(N)=\int_{\mathbb{R}^k}\int_{[0,1)^s}e\left(\beta_k\cdot L_k(\bfgamma)+\ldots+\beta_1\cdot L_1(\bfgamma)\right)\, d\bfgamma \, d\bfbeta.
\end{displaymath}
We know by Lemma \ref{lem7} that there exists a non-singular real solution $(\xi_1,\ldots,\xi_s)\in(0,1)^s$ to the system (\ref{equation1}).  By the Implicit Function Theorem (see, for example, Theorem 13.7 in \cite{apostol}), there is an $(s-k)$-dimensional subspace $\mathcal{S}\subset[0,1)^s$ of positive volume containing the solution $(\xi_1,\ldots,\xi_s)$ in which every point satisfies the system (\ref{equation1}).  Therefore, by applying $k$ iterations of the Fourier integral formula in the form
\begin{displaymath}
\lim_{\Omega\rightarrow\infty}\int_{-T}^T\int_{-\Omega}^{\Omega}F(t)e(t\omega)\, d\omega\, dt=F(0)
\end{displaymath}
to the integral $\mathcal{C}(N)$, we deduce that
\begin{displaymath}
\mathcal{C}(N)=\int_{\mathcal{S}}\, d\mathcal{S}=mes(\mathcal{S})>0.
\end{displaymath}
Setting $\mathcal{C}=mes(\mathcal{S})$ gives the desired result.
\end{proof}

To prove that $\mathfrak{S}(I)=\mathfrak{S}>0$ for some constant $\mathfrak{S}$, we require the following lemmas.

\begin{lemma}
For each prime $p$, define
\begin{displaymath}
T(p)=\sum_{h=0}^{\infty}S(p^h).
\end{displaymath}
Then $T(p)$ and $\prod_p T(p)$ converge absolutely, and $\mathfrak{S}(I)=\prod_p T(p)$.  Moreover, there is a positive constant $p_0$ depending at most on $\lambda_1,\ldots,\lambda_s$ such that
\begin{displaymath}
\frac{1}{2}<\prod_{p\geq p_0}T(p)<\frac{3}{2}.
\end{displaymath}
\label{lem13}
\end{lemma}
\begin{proof}
The absolute convergence of $T(p)$ follows from the upper bound on $S(q)$ in (\ref{equation14}).  The other results follow from Theorem 2.4 of \cite{vaughan}.
\end{proof}

\begin{lemma}
For any $q\in\mathbb{N}$, we have
\begin{displaymath}
\sum_{d|q}S(d)=q^{k-s}M_n(q).
\end{displaymath}
\label{lem5}
\end{lemma}
\begin{proof}
Recalling (\ref{equation27}), this is a slight modification of Lemma 2.12 in \cite{vaughan}.
\end{proof}

By Lemma \ref{lem5} and the definition of $T(p)$, we have
\begin{displaymath}
T(p)=\lim_{t\rightarrow\infty} p^{(k-s)t} M_n(p^t),
\end{displaymath}
provided this limit exists.  It follows by Lemma \ref{lem11} that
\begin{displaymath}
T(p)\geq p^{u(p)\cdot(k-s)}>0,
\end{displaymath}
and so by Lemma \ref{lem13} we have $\mathfrak{S}>0$.  Combining this result with Lemmas \ref{lem10} and \ref{lem14}, we deduce the following lemma.

\begin{lemma}
The approximation integral $I$ satisfies the lower bound
\begin{displaymath}
I\gg\mathcal{C}\mathfrak{S}N^{s-k(k+1)/2},
\end{displaymath}
where $\mathcal{C}$ and $\mathfrak{S}$ are positive constants depending at most on the $\lambda_i$.
\label{lem2}
\end{lemma}

\vspace{4mm}

\section{Putting everything together}
In the previous sections, we obtained an upper bound for the error $E(\bfalpha)$, an upper bound for the moment integral $J$, and a lower bound for the approximation integral $I$.  We are therefore ready to combine these estimates to obtain a lower bound on the number of solutions $\mathcal{N}$ in $\mathcal{A}_N^s$ to the system (\ref{equation1}) when the $k$-degree uniformity parameter $\mathfrak{a}$ is suitably bounded in terms of the density $\delta_N$.  If $\mathfrak{a}$ does not obey this bound, then we show that we may find a proper arithmetic progression $\mathcal{P}\subset[1,N]$ such that the density of $\mathcal{A}$ in $\mathcal{P}$ is larger than $\delta_N$.  We show that we may iterate this process until either we reduce to the uniform case or the densities of the new sets approach 1.

We begin by combining the results of Lemmas \ref{lem9}, \ref{lem8}, and \ref{lem2} to re-write the inequality in (\ref{equation11}) as
\begin{equation}
\mathcal{C}\mathfrak{S}\delta_N^s N^{s-k(k+1)/2}-\mathcal{N}\ll 2\mathfrak{a}^{2^{-k-1}}\cdot\delta_N^{s-1-s_0(k)}N^{s-k(k+1)/2}.
\label{equation4}
\end{equation}
Suppose that $\mathfrak{a}\leq(\mathcal{C}\mathfrak{S}\delta_N^{s_0(k)+2}/4)^{2^{k+1}}$.  Substituting this bound into (\ref{equation4}) yields
\begin{displaymath}
\mathcal{C}\mathfrak{S}\delta_N^s N^{s-k(k+1)/2}-\mathcal{N}\ll \frac{\mathcal{C}\mathfrak{S}}{2}\delta_N^{s+1}N^{s-k(k+1)/2},
\end{displaymath}
or, equivalently,
\begin{displaymath}
\mathcal{C}\mathfrak{S}\delta_N^s N^{s-k(k+1)/2}\cdot\left(1-\delta_N/2\right)\ll\mathcal{N}.
\end{displaymath}
We note that, as $0\leq\delta_N\leq 1$, we have $1/2\leq 1-\delta_N/2 \leq 1$, and therefore
\begin{displaymath}
\frac{\mathcal{C}\mathfrak{S}}{2}\cdot\delta_N^s N^{s-k(k+1)/2}\ll\mathcal{N}.
\end{displaymath}
Theorem \ref{thm2} now follows by taking $K=(\mathcal{C}\mathfrak{S}/4)^{2^{k+1}}$ in (\ref{equation34}).

The contrapositive of this result implies that if $\mathcal{A}_N^s$ contains no non-trivial solutions to the system (\ref{equation1}), then the set $\mathcal{A}_N$ is not $\mathfrak{a}$-uniform of degree $k$ for $\mathfrak{a}$ satisfying the bound stated in Theorem \ref{thm2}.  We note that since the uniformity parameter $\mathfrak{a}$ is bounded above by 1, we may assume henceforth that
\begin{equation}
\delta_N<\left(\frac{\mathcal{C}\mathfrak{S}}{4}\right)^{-1/(s_0(k)+2)}.
\label{equation39}
\end{equation}
We observe that this bound is trivial if the right-hand side exceeds 1.

We now recall the following result from \cite{gowers2}.

\begin{lemma}
If $\mathcal{A}_N$ is not $\mathfrak{a}$-uniform of degree $k$ for a given value of $\mathfrak{a}$, then there is a proper arithmetic progression $\mathcal{P}$ of cardinality at least $N^{\mathfrak{b}}$, where $\mathfrak{b}=\mathfrak{a}^{2^{2^{k+8}}}$, such that
\begin{displaymath}
\left|\mathcal{A}_N\cap\mathcal{P}\right|\geq\left|\mathcal{P}\right|\left(\delta_N+\mathfrak{b}\right).
\end{displaymath}
\label{lem6}
\end{lemma}
\begin{proof}
We first embed the set $\mathcal{A}_N\subset [1,N]$ into $\mathbb{Z}/N\mathbb{Z}$, calling the new set $\mathcal{A}_N^*$.  Recalling the discussion of uniformity of degree $k$ surrounding the definition (\ref{equation7}), we note that the hypothesis that $\mathcal{A}_N$ is not $\mathfrak{a}$-uniform of degree $k$ for a given value of $\mathfrak{a}$ implies that $\mathcal{A}_N^*$ is not $\mathfrak{a}$-uniform of degree $k$ in the sense of Gowers.  We may therefore suppose that $\mathcal{A}_N^*$ is not $\mathfrak{a}$-uniform of degree $k$ for the given value of $\mathfrak{a}$.

By the proof of Theorem 18.2 in \cite{gowers2}, we know that if $\mathcal{A}_N^*$ is not $\mathfrak{a}$-uniform of degree $k$ for some value of $\mathfrak{a}$, then there exists a proper arithmetic progression $\mathcal{P}^*$ of cardinality at least $N^{\mathfrak{b}}$ such that
\begin{displaymath}
\left|\mathcal{A}_N^*\cap\mathcal{P}^*\right|\geq\left|\mathcal{P}^*\right|\left(\delta_N+\mathfrak{b}\right).
\end{displaymath}
As $\mathcal{P}^*$ is a proper arithmetic progression, it may be associated to an identical arithmetic progression $\mathcal{P}\subset[1,N]$ such that
\begin{displaymath}
\left|\mathcal{A}_N\cap\mathcal{P}\right|\geq\left|\mathcal{P}\right|\left(\delta_N+\mathfrak{b}\right).
\end{displaymath}
This $\mathcal{P}$ is the progression we seek.
\end{proof}

Define
\begin{displaymath}
\gamma(k)=2^{2^{k+8}+k+1},\qquad \mathcal{K}=\left(\frac{\mathcal{C}\mathfrak{S}}{4}\right)^{\gamma(k)},\qquad C=\left(s_0(k)+2\right)\gamma(k).
\end{displaymath}
It follows by Lemma \ref{lem6} that if the set $\mathcal{A}_N$ fails to be $\mathfrak{a}$-uniform of degree $k$ for any value of $\mathfrak{a}\leq(\mathcal{C}\mathfrak{S}\delta_N^{s_0(k)+2}/4)^{2^{k+1}}$, then there exists a proper arithmetic progression $\mathcal{P}_1$ such that the cardinality of $\mathcal{P}_1$ obeys the lower bound
\begin{displaymath}
\left|\mathcal{P}_1\right|\geq N^{\mathcal{K}\delta_N^C},
\end{displaymath}
and such that the density of $\mathcal{A}_N$ in $\mathcal{P}_1$ obeys the lower bound
\begin{displaymath}
\frac{\left|\mathcal{A}_N\cap\mathcal{P}_1\right|}{\left|\mathcal{P}_1\right|}\geq \delta_N+\mathcal{K}\delta_N^C.
\end{displaymath}
Let $N_0=N$ and $N_1=|\mathcal{P}_1|$.  Since $\mathcal{P}_1$ is a translation and dilation of the interval $[1,N_1]$, we may map $\mathcal{P}_1$ to the interval $[1,N_1]$ by reversing the translation and dilation.  This process maps the intersection $\mathcal{A}_{N_0}\cap\mathcal{P}_1$ to a subset $\mathcal{A}_{N_1}$ of $[1,N_1]$, the density $\delta_{N_1}$ of which is
\begin{displaymath}
\delta_{N_1}=\frac{\left|\mathcal{A}_{N_1}\right|}{N_1}=\frac{\left|\mathcal{A}_{N_0}\cap\mathcal{P}_1\right|}{\left|\mathcal{P}_1\right|}.
\end{displaymath}

We now iterate this process.  After $r+1$ iterative steps, we are concerned with the set $\mathcal{A}_{N_r}\subset[1,N_r]$, which has density $\delta_{N_r}=|\mathcal{A}_{N_r}|/N_r$.  If $\mathcal{A}_{N_r}$ is not $\mathfrak{a}$-uniform of degree $k$ for any value of $\mathfrak{a}$ satisfying the upper bound
\begin{displaymath}
\mathfrak{a}\leq\left(\frac{\mathcal{C}\mathfrak{S}}{4}\delta_{N_r}^{s_0(k)}\right)^{2^{k+1}},
\end{displaymath}
then we may find a proper arithmetic progression $\mathcal{P}_{r+1}\subset[1,N_r]$ of cardinality $N_{r+1}$, where
\begin{displaymath}
N_{r+1}\geq N_r^{\mathcal{K}\delta_{N_r}^C},
\end{displaymath}
such that the density $\delta_{N_{r+1}}$ of $\mathcal{A}_{N_r}$ in $\mathcal{P}_{r+1}$ obeys the lower bound
\begin{displaymath}
\delta_{N_{r+1}}\geq\delta_{N_r}+\mathcal{K}\delta_{N_r}^C.
\end{displaymath}
Since $\delta_{N_{r+1}}\geq\delta_{N_r}$ for all $r\geq 0$, it follows that
\begin{displaymath}
\delta_{N_{r+1}}\geq\delta_{N_r}+\mathcal{K}\delta_{N_0}^C\geq\delta_N+r\mathcal{K}\delta_N^C,
\end{displaymath}
where we recall that we have defined $N_0=N$.  We may therefore perform at most $(\mathcal{K}\delta_N^C)^{-1}$ iterations of the concentration process before the density reaches 1.  Moreover, at each step of the iterative process, the size of the ambient set is raised to a power of at least $\mathcal{K}\delta_N^C$.

Let
\begin{displaymath}
Y=\inf_{\bfy}\left(\max_{1\leq i\leq s} y_i\right),
\end{displaymath}
where the infimum is taken over all non-trivial solutions $\bfy\in\mathbb{N}^s$ to the system (\ref{equation1}).  We note that if $Y=1$, then the system (\ref{equation1}) has no non-trivial solutions, while if $Y=2$, then any set containing more than two elements necessarily furnishes non-trivial solutions to the system (\ref{equation1}) by the property of translation and dilation invariance.  We may therefore assume that $Y\geq 3$.

Now, if the densities of the sets $\mathcal{A}_{N_r}$ in the iterative process reach 1 before the size of the ambient sets $[1,N_r]$ becomes less than $Y$, then the original set $\mathcal{A}_N$ necessarily furnishes a non-trivial solution to the system (\ref{equation1}).  Define
\begin{displaymath}
D=\mathcal{K}\delta_N^C.
\end{displaymath}
We note that by our assumption (\ref{equation39}) on the size of $\delta_N$, we must have $D\leq 1$.  Suppose that
\begin{equation}
\delta_N>\left(\frac{1}{2\mathcal{K}}\right)^{1/C},
\label{equation41}
\end{equation}
so that $D>1/2$.  Then after two iterations of the concentration process, we obtain an arithmetic progression $\mathcal{P}_2$ of length at least $N^{1/4}$ such that the density of $\mathcal{A}_N$ in $\mathcal{P}_2$ is 1.  Hence, if $N\geq Y^4$ and $\delta_N$ satisfies the bound in (\ref{equation41}), the set $\mathcal{A}_N$ necessarily furnishes a non-trivial solution to the system (\ref{equation1}).

We now suppose that (\ref{equation41}) does not hold, so that $D<1/2$.  We will obtain a non-trivial solution in $\mathcal{A}_N^s$ to the system (\ref{equation1}) provided that
\begin{displaymath}
N^{D^{D^{-1}}}\geq Y,
\end{displaymath}
or, equivalently,
\begin{displaymath}
N\geq Y^{D^{-D^{-1}}}.
\end{displaymath}
Taking the logarithm of both sides twice, we obtain
\begin{displaymath}
\log\log N\geq-\frac{\log D}{D}+\log\log Y>-\frac{\log D}{D},
\end{displaymath}
where the second inequality follows from the assumption that $Y\geq 3$.  As we have assumed that $D<1/2$, we have
\begin{displaymath}
\log\log N>\frac{1}{2D}.
\end{displaymath}
By the definition of $D$, this is equivalent to
\begin{displaymath}
\log\log N\gg\delta_N^{-C},
\end{displaymath}
where the implied constants depend on the $\lambda_i$ and $k$.

The contrapositive of this result implies that if $\mathcal{A}_N$ does not furnish any non-trivial solutions to the system (\ref{equation1}), then $\delta_N$ must satisfy
\begin{displaymath}
\delta_N\ll(\log\log N)^{-1/C}.
\end{displaymath}
We may verify by a straightforward computation that $C<2^{2^{k+9}}$, whence we have
\begin{displaymath}
\delta_N\ll(\log\log N)^{-2^{-2^{k+9}}}.
\end{displaymath}
Theorem \ref{thm1} therefore follows by taking $c=2^{-2^{k+9}}$.

\end{document}